\documentclass[a4paper]{amsart}
\usepackage[latin1]{inputenc}
\usepackage{amssymb}
\usepackage{amsmath}
\usepackage{mathrsfs}
\usepackage{amsthm}
\usepackage{amsfonts}
\usepackage{textcomp}
\usepackage{graphicx}
\usepackage[pdftex]{color}
\usepackage{paralist}
\usepackage[shortlabels]{enumitem}
\usepackage{hyperref}
\usepackage{comment}
\usepackage{fvextra}
\usepackage[arrow, matrix, curve]{xy}
\usepackage{tikz}
\DefineVerbatimEnvironment{MagmaCode}{Verbatim}{%
  fontsize=\scriptsize,%
  baselinestretch=0.88,%
  breaklines=true,%
  breakanywhere=true,%
  xleftmargin=0.5em,%
  xrightmargin=0.5em%
}

\newtheorem*{corollary*}{Corollary}
\newtheorem*{conjecture*}{Conjecture}
\newtheorem*{example*}{Example}
\newtheorem*{theorem*}{Theorem}
\newtheorem*{proposition*}{Proposition}

\newtheorem{theorem}{Theorem}[section]

\newtheorem*{claim*}{Claim}
\newtheorem*{conjecture}{Conjecture}
\newtheorem*{question}{Question}

\theoremstyle{definition}

\theoremstyle{remark}

\numberwithin{equation}{section}

\makeatletter
\renewcommand*\env@matrix[1][\arraystretch]{%
  \edef\arraystretch{#1}%
  \hskip -\arraycolsep
  \let\@ifnextchar\new@ifnextchar
  \array{*\c@MaxMatrixCols c}}
\makeatother

\renewcommand{\mod}{\operatorname{mod}}

\newcommand{\Ext}{\operatorname{Ext}}

\newcommand{\End}{\operatorname{End}}

\newcommand{\Hom}{\operatorname{Hom}}

\newcommand{\Tor}{\operatorname{Tor}}

\begin{document}

\title{Tor and Ext vanishing results for commutative Artinian rings}
\date{\today}



\author{Bernhard B\"ohmler}
\address[B.~B\"ohmler]{Leibniz Universit\"at Hannover, Institut f\"ur Algebra, Zahlentheorie und Diskrete Mathematik, Welfengarten 1, 30167 Hannover, Germany}
\email{boehmler@math.uni-hannover.de}

\author{Ren\'{e} Marczinzik}
\address[R.~Marczinzik]{Mathematical Institute of the University of Bonn, Endenicher Allee 60, 53115 Bonn, Germany}
\email{marczire@math.uni-bonn.de}

\begin{abstract}
We give a negative answer to a question of Avramov, Buchweitz and \c{S}ega by constructing a commutative local finite-dimensional non-Gorenstein algebra $R$ with $\Ext_R^1(D(R),R)=0$; this question is related to the first Tachikawa conjecture. We also give a counterexample to a conjecture of Huneke, \c{S}ega and Vraciu on Tor vanishing over commutative finite-dimensional algebras. Finally, we construct a finite-dimensional commutative local self-injective algebra $R$ over $\mathbb{F}_2$ and an indecomposable non-projective $R$-module $M$ such that $\Ext_R^1(M,M)=\Ext_R^2(M,M)=0$, related to the second Tachikawa conjecture and answering a question of Dao.
\end{abstract}

\maketitle

\section*{Introduction}
In the book \cite{Ta}, Tachikawa formulated two homological conjectures for finite-dimensional, not necessarily commutative, $K$-algebras ($K$ a field) that are closely related to the famous Nakayama conjecture.
The Tachikawa conjectures remain open in general and are known only in special situations. They follow from the celebrated finitistic dimension conjecture; see, for example, the survey \cite{Y}.
The first Tachikawa conjecture states that, for a finite-dimensional $K$-algebra $A$ with dual $D(A)=\Hom_K(A,K)$, the vanishing $\Ext_A^i(D(A),A)=0$ for all $i\geq1$ implies that $A$ is self-injective. The second Tachikawa conjecture states that, if $A$ is self-injective and $M$ is an indecomposable $A$-module satisfying $\Ext_A^i(M,M)=0$ for all $i\geq1$, then $M$ is projective.
Both conjectures remain open even for finite-dimensional commutative algebras, but several commutative cases are known. For a proof of the conjectures over commutative local Artinian algebras with radical cube zero, see Asashiba and Hoshino \cite{A,H,H2}. For broader background on the noncommutative conjectures, see \cite{CFX,Sa}.

In positive Krull dimension, the natural commutative analogue of the first conjecture replaces the vector-space dual by a canonical module in the Cohen--Macaulay case, or more generally by a dualizing complex. Avramov, Buchweitz and \c{S}ega formulated this version and proved it in several classes of commutative Noetherian local rings \cite{ABS}. Further work treats Cohen--Macaulay tensor products \cite{Jor}, generically Gorenstein Cohen--Macaulay rings \cite{GT}, and positively graded rings over fields of characteristic different from two \cite{LMSW}. Related Ext-vanishing criteria for the Gorenstein property were obtained by Hanes and Huneke \cite{HH}; Lyle and Monta\~no later established further criteria based on consecutive blocks of Ext vanishing \cite{LM}. A nearly Gorenstein variant, formulated via the trace of the canonical module, was posed by Dao, Kobayashi and Takahashi \cite{DKT}; more recent annihilator-theoretic generalizations and results in Krull dimensions one, two and higher appear in \cite{Es}.

The resulting dualizing-complex conjecture is the following:
\begin{conjecture}
Let $R$ be a commutative Noetherian local ring with dualizing complex $D^R$. If $\Ext_R^i(D^R,R)=0$ for all $i \geq 1$, then $R$ is Gorenstein.
\end{conjecture}
We refer to the introduction of \cite{ABS} for this conjecture and the precise definition of $D^R$.
In the introduction of \cite{ABS}, one can find the following question:
\begin{question}
Let $R$ be a commutative Noetherian local ring with dualizing complex $D^R$. If $\Ext_R^i(D^R,R)=0$ for $\dim R+1$ consecutive positive values of $i$, must $R$ be Gorenstein?
\end{question}
Note that, when $R$ is a commutative local finite-dimensional $K$-algebra, it is Cohen--Macaulay and $D^R=\Hom_K(R,K)=D(R)$.
We give a negative answer as follows:
\begin{theorem}
There exists a finite-dimensional commutative local non-Gorenstein $\mathbb{Q}$-algebra $R$ such that $\Ext_R^1(D(R),R)=0$.
\end{theorem}
This example was found with the assistance of ChatGPT and is rather complicated: it has 18 generators and many relations. We give its explicit presentation in Section~2, together with Magma code verifying the theorem.

Next we consider the second Tachikawa conjecture for commutative local finite-dimensional Gorenstein algebras. For commutative Artinian local rings, being Gorenstein is equivalent to being self-injective.
In \cite[Question 9.1.4]{D}, Dao posed the following question:
\begin{question}
Let $R$ be a commutative local finite-dimensional Gorenstein algebra with a finitely generated module $M$. If $\Ext_R^1(M,M)=0$, is $M$ free?
\end{question}
Note that over local finite dimensional algebras, being free is equivalent to being projective for finitely generated modules.
For commutative Artinian Gorenstein rings, the second Tachikawa conjecture is equivalent to the commutative Auslander--Reiten conjecture, which is discussed for example, in \cite{S}. In positive Krull dimension, the latter asks whether a finitely generated module $M$ over a commutative Noetherian local ring $R$ satisfying
\[
 \Ext_R^i(M,M\oplus R)=0\qquad\text{for all }i\geq1
\]
must be free. Auslander, Ding and Solberg proved the conjecture for complete intersections \cite{ADS}, and Huneke and Leuschke proved it for locally excellent Cohen--Macaulay normal rings containing the rational numbers \cite{HL}. It was subsequently established for Cohen--Macaulay normal rings by Kimura, Otake and Takahashi \cite{KOT}, and more recently for every normal ring by Kimura \cite{Ki}. Results for Gorenstein rings and extensions to Cohen--Macaulay local rings with canonical modules appear in \cite{Ar,OY,ACST}; the rank-one maximal Cohen--Macaulay case is treated in \cite{GT}. Further numerical cases and generalizations based on Ext annihilators are developed in \cite{LM,Es}.

Our next main result gives a negative answer to Dao's question:
\begin{theorem}
There is a finite-dimensional commutative local Gorenstein $\mathbb{F}_2$-algebra $R$ admitting an indecomposable non-projective $R$-module $M$ such that $\Ext_R^i(M,M)=0$ for $i=1,2$.
\end{theorem}
We display the concrete algebra $R$ and module $M$, which were also found with the assistance of ChatGPT, in Section~3, together with Magma code verifying the theorem.
This also resolves a problem in the literature that had remained open for nearly 20 years.
In \cite[Theorem~4.6]{LH} a result was presented that would imply that if $R$ is a commutative local finite-dimensional Gorenstein algebra and $M$ is an indecomposable $R$-module satisfying $\Ext_R^i(M,M)=0$ for $i=1,2$, then $M$ is projective. This would imply the second Tachikawa conjecture in the commutative case. Doubts about the argument have been expressed for a long time, but no concrete counterexample was previously known; see, for example, the footnote on the first page of \cite{S}. The theorem above supplies such a counterexample and shows that the argument in \cite[Theorem~4.6]{LH} cannot be correct.
We remark that our previous two theorems also answer questions raised in \cite[Questions 2.4 and 2.5]{MO}.

The last homological problem that we look at in this article is the following conjecture by Huneke, \c{S}ega and Vraciu in \cite{HSV}:
\begin{conjecture}
Let $R$ be a commutative Artinian local ring with maximal ideal $\mathfrak{m}$. Assume that there are two nonzero finitely generated $R$-modules $M$ and $N$ with $\mathfrak{m}^2M=0$ and $\mathfrak{m}^2N=0$. If $\Tor_i^R(M,N)=0$ for all $i \geq 1$, then $\mathfrak{m}^3=0$.
\end{conjecture}

The next theorem gives a counterexample to this conjecture:
\begin{theorem}
Let $K$ be a field, let $R=K[x,y]/(xy,x^2-y^3)$, set $\mathfrak m=(x,y)R$, and let $M=R/(y)$ and $N=R/(x-y)$. Then $\mathfrak{m}^2M=0$, $\mathfrak{m}^2N=0$, and $\Tor_i^R(M,N)=0$ for all $i\geq1$, whereas $\mathfrak{m}^3\neq0$.
\end{theorem}
The algebra $R$ was also found with the assistance of ChatGPT.
The algebra $R$ is five-dimensional, and in this case we give a direct verification in Section~1.

For the computer verifications, the reader may copy and paste the relevant code into the Magma online calculator if the commercial version of Magma is unavailable:
\url{https://magma.maths.usyd.edu.au/calc/}.
We assume the reader is familiar with commutative algebra and refer for example to the textbook \cite{E}.
\section{A counterexample to the conjecture of Huneke, \c{S}ega and Vraciu}

In this section we prove the following theorem:
\begin{theorem}
Let $K$ be a field and set
\[
 R=K[x,y]/(xy,x^2-y^3),\qquad
 M=R/(y),\qquad
 N=R/(x-y).
\]
Let $\mathfrak m=(x,y)R$ be the maximal ideal of $R$.  Then
\[
 \mathfrak m^2M=0,\qquad
 \mathfrak m^2N=0,\qquad
 \Tor_i^R(M,N)=0\quad\text{for every }i\geq1,
\]
whereas $\mathfrak m^3\neq0$.
\end{theorem}
\begin{proof}
Put
\[
 P=K[X,Y],\qquad I=(XY,X^2-Y^3)\subseteq P,
\]
and let
\[
 R=P/I.
\]
We denote the residue classes of $X$ and $Y$ in $R$ by $x$ and $y$,
respectively.  Thus the defining relations in $R$ are
\begin{equation}\label{eq:defining-relations}
 xy=0,
 \qquad
 x^2=y^3.
\end{equation}
We shall verify all assertions in the theorem directly.

\medskip
\noindent
\textbf{Step 1: the vector-space structure of $R$, its maximal ideal,
and its Loewy length.}

We first determine a $K$-basis of $R$.
Consider the intermediate quotient
\[
 S=P/(X^2-Y^3),
\]
and write $\widetilde x$ and $\widetilde y$ for the images of $X$ and
$Y$ in $S$.  Regard $P=K[Y][X]$ as a polynomial ring in the one
variable $X$ with coefficients in $K[Y]$.  The polynomial
$X^2-Y^3$ is monic as a polynomial in $X$.  Hence the ordinary
one-variable division algorithm shows that every polynomial
$p\in P$ can be written uniquely in the form
\[
 p=q(X,Y)(X^2-Y^3)+f(Y)+Xg(Y),
 \qquad f(Y),g(Y)\in K[Y].
\]
Consequently, every element of $S$ has a unique expression
\begin{equation}\label{eq:S-normal-form}
 f(\widetilde y)+\widetilde x g(\widetilde y),
 \qquad f,g\in K[Y].
\end{equation}

The ring $R$ is obtained from $S$ by imposing the additional relation
$\widetilde x\widetilde y=0$.  We now describe the ideal generated by
$\widetilde x\widetilde y$ in $S$.  Using the unique expression in
\eqref{eq:S-normal-form}, an arbitrary multiple of
$\widetilde x\widetilde y$ has the form
\begin{align*}
 \bigl(f(\widetilde y)+\widetilde xg(\widetilde y)\bigr)
 \widetilde x\widetilde y
 &=\widetilde x\widetilde y f(\widetilde y)
   +\widetilde x^2\widetilde y g(\widetilde y)\\
 &=\widetilde x\widetilde y f(\widetilde y)
   +\widetilde y^4g(\widetilde y),
\end{align*}
because $\widetilde x^2=\widetilde y^3$ in $S$.  Conversely, every
sum of a multiple of $\widetilde y^4$ and a multiple of
$\widetilde x\widetilde y$ occurs in this way.  The uniqueness in
\eqref{eq:S-normal-form} also
shows that the two types of summands cannot cancel each other.  Thus,
as a $K$-vector space,
\[
 (\widetilde x\widetilde y)S
 =\widetilde y^4K[\widetilde y]
  \mathbin{\oplus}
  \widetilde x\widetilde yK[\widetilde y].
\]
It follows that every element of
$R=S/(\widetilde x\widetilde y)$ has a unique representative of the
form
\[
 a+bx+cy+dy^2+ey^3,
 \qquad a,b,c,d,e\in K.
\]
Therefore
\begin{equation}\label{eq:R-basis}
 1,\quad x,\quad y,\quad y^2,\quad y^3
\end{equation}
is a $K$-basis of $R$.  In particular,
\[
 \dim_K R=5
 \qquad\text{and}\qquad
 y^3\neq 0.
\]

We next record some immediate consequences of the defining relations.
Multiplying $x^2=y^3$ by $y$ gives
\[
 x^2y=y^4.
\]
The left-hand side is $x(xy)=0$.  Hence
\begin{equation}\label{eq:y4zero}
 y^4=0.
\end{equation}
Also,
\[
 x^3=x\cdot x^2=xy^3=y^2(xy)=0.
\]
More generally, every monomial containing both a positive power of
$x$ and a positive power of $y$ is zero, because it contains the
factor $xy$.

Let
\[
 \mathfrak m=(x,y)R.
\]
By the basis in \eqref{eq:R-basis}, the ideal $\mathfrak m$ is the
$K$-vector space
\begin{equation}\label{eq:m-basis}
 \mathfrak m=Kx\oplus Ky\oplus Ky^2\oplus Ky^3.
\end{equation}

The ideal $\mathfrak m^2$ is generated by the pairwise products of
$x$ and $y$.  Therefore, using \eqref{eq:defining-relations},
\[
 \mathfrak m^2=(x^2,xy,y^2)=(y^3,0,y^2)=(y^2,y^3).
\]
We now identify this ideal as a vector space.  If
$r=a+bx+cy+dy^2+ey^3$, then
\[
 ry^2=ay^2+cy^3
 \qquad\text{and}\qquad
 ry^3=ay^3,
\]
because $xy=0$ and $y^4=0$.  Thus
\begin{equation}\label{eq:m2-basis}
 \mathfrak m^2=Ky^2\oplus Ky^3.
\end{equation}
Multiplying once more by $\mathfrak m$, we use
\[
 xy^2=xy^3=0,
 \qquad
 y\cdot y^2=y^3,
 \qquad
 y\cdot y^3=y^4=0,
\]
to obtain
\begin{equation}\label{eq:m3-basis}
 \mathfrak m^3=Ky^3.
\end{equation}
Since $y^3\neq0$, we have $\mathfrak m^3\neq0$.  Finally,
\[
 \mathfrak m^4=\mathfrak m\mathfrak m^3=0,
\]
because both $xy^3$ and $y^4$ are zero.  Thus $\mathfrak m$ has nilpotency index four.

Since $\dim_K R=5$, the ring $R$ is Artinian.  To prove that it is local,
evaluate at $X=Y=0$.  This induces a surjective ring homomorphism
\[
 \pi:R\longrightarrow K,
 \qquad
 \pi(a+bx+cy+dy^2+ey^3)=a.
\]
Its kernel is exactly $\mathfrak m$, by \eqref{eq:m-basis}.  Hence
\[
 R/\mathfrak m\cong K,
\]
so $\mathfrak m$ is maximal.  Since $\mathfrak m^4=0$, every maximal ideal contains $\mathfrak m$; therefore $\mathfrak m$ is the unique maximal ideal of $R$.  Consequently, $R$ is local and has Loewy length four.

\medskip
\noindent
\textbf{Step 2: the modules $M$ and $N$, and the action of
$\mathfrak m^2$.}

The modules $M=R/(y)$ and $N=R/(x-y)$ are cyclic quotient modules, so
they are finitely generated. They are also quotients of $R$ by an ideal and thus $K$-algebras. We now describe them explicitly via their $K$-algebra and $R$-module structure and in particular show that they are 2-dimensional as a vector space over $K$. As they have simple top, they are  indecomposable.

In $M=R/(y)$, the element $y$ acts as zero, and the relation
$x^2=y^3$ therefore becomes $x^2=0$.  The assignment
\[
 t\longmapsto x+(y)
\]
defines a surjective $K$-algebra homomorphism
\[
 K[t]/(t^2)\longrightarrow M.
\]
Conversely, the assignments $x\mapsto t$ and $y\mapsto0$ define a
homomorphism $R\to K[t]/(t^2)$: indeed, both $xy$ and
$x^2-y^3$ are sent to zero.  This homomorphism kills the ideal $(y)$,
so it factors through $M$.  The two induced maps are inverse to one
another because they agree with the identity on the generators.
Consequently,
\begin{equation}\label{eq:M-dual-numbers}
 M\cong K[t]/(t^2).
\end{equation}
In particular, $M$ is nonzero and has $K$-basis $1,t$.

For $N=R/(x-y)$, let
\[
 u=x+(x-y)=y+(x-y)
\]
be the common image of $x$ and $y$.  The relation $xy=0$ gives
$u^2=0$.  Therefore the assignment $u\mapsto x+(x-y)$ gives a
surjective homomorphism
\[
 K[u]/(u^2)\longrightarrow N.
\]
Conversely, sending both $x$ and $y$ to $u$ defines a homomorphism
$R\to K[u]/(u^2)$, because
\[
 xy\longmapsto u^2=0
 \qquad\text{and}\qquad
 x^2-y^3\longmapsto u^2-u^3=0.
\]
It also sends $x-y$ to zero, so it factors through $N$.  Again the two
induced maps are inverse on the generators.  Hence
\begin{equation}\label{eq:N-dual-numbers}
 N\cong K[u]/(u^2).
\end{equation}
Thus $N$ is nonzero and has $K$-basis $1,u$.

By \eqref{eq:m2-basis}, the ideal $\mathfrak m^2$ is spanned by $y^2$
and $y^3$.  On $M$, the element $y$ acts as zero, so both $y^2$ and
$y^3$ act as zero.  On $N$, the element $y$ acts as $u$, and
$u^2=0$, so again both $y^2$ and $y^3$ act as zero.  Therefore
\begin{equation}\label{eq:m2-kills-modules}
 \mathfrak m^2M=0
 \qquad\text{and}\qquad
 \mathfrak m^2N=0.
\end{equation}
Moreover, $\mathfrak mM\neq0$ because $x\cdot1=t\neq0$ in $M$, and
$\mathfrak mN\neq0$ because $x\cdot1=u\neq0$ in $N$.  Thus both
modules have Loewy length exactly two.

\medskip
\noindent
\textbf{Step 3: the annihilators of $x$ and $y$.}

Every element $r\in R$ has a unique expression
\[
 r=a+bx+cy+dy^2+ey^3,
 \qquad a,b,c,d,e\in K.
\]
Multiplication by $y$ gives
\begin{align*}
 ry
 &=ay+bxy+cy^2+dy^3+ey^4\\
 &=ay+cy^2+dy^3.
\end{align*}
The elements $y,y^2,y^3$ are linearly independent by
\eqref{eq:R-basis}.  Hence $ry=0$ if and only if
$a=c=d=0$.  Therefore
\[
 (0:_R y)=Kx\oplus Ky^3.
\]
On the other hand, if
$s=\alpha+\beta x+\gamma y+\delta y^2+\varepsilon y^3$, then
\[
 sx=\alpha x+\beta y^3.
\]
Therefore every element of $Kx\oplus Ky^3$ occurs in this way.  Thus
\[
 xR=Kx\oplus Ky^3=(0:_R y).
\]

Similarly,
\begin{align*}
 rx
 &=ax+bx^2+cyx+dy^2x+ey^3x\\
 &=ax+by^3.
\end{align*}
The elements $x$ and $y^3$ are linearly independent, so $rx=0$ if and
only if $a=b=0$.  It follows that
\[
 (0:_R x)=Ky\oplus Ky^2\oplus Ky^3.
\]
But multiplication of a general element of $R$ by $y$ shows that
\[
 yR=Ky\oplus Ky^2\oplus Ky^3.
\]
Consequently,
\begin{equation}\label{eq:annihilators}
 (0:_R y)=xR=(x),
 \qquad
 (0:_R x)=yR=(y).
\end{equation}
Together with $xy=0$, these identities say that $x$ and $y$ form an
exact pair of zero divisors.

\medskip
\noindent
\textbf{Step 4: an explicit two-periodic free resolution of $M$.}

Let $F_i=R$ for every $i\geq0$, and let
\[
 \varepsilon:F_0=R\longrightarrow M=R/(y)
\]
be the quotient map.  For $i\geq1$, define an $R$-linear map
$d_i:F_i\to F_{i-1}$ by
\[
 d_i(r)=
 \begin{cases}
  ry,&\text{if $i$ is odd},\\
  rx,&\text{if $i$ is even}.
 \end{cases}
\]
Because $xy=0$, the composite of any two consecutive maps is zero.
Thus these maps form the augmented complex
\begin{equation}\label{eq:resolution}
 \cdots\xrightarrow{\,x\,}R
 \xrightarrow{\,y\,}R
 \xrightarrow{\,x\,}R
 \xrightarrow{\,y\,}R
 \xrightarrow{\,\varepsilon\,}M
 \longrightarrow0.
\end{equation}
We verify exactness at every term.

At $F_0$, the kernel of $\varepsilon$ is the ideal $(y)=yR$, while the
image of $d_1$ is also $yR$.  Hence the complex is exact at $F_0$.
Now let $i\geq1$.
If $i$ is odd, then $d_i$ is multiplication by $y$, so
\[
 \ker d_i=(0:_R y)=xR
\]
by \eqref{eq:annihilators}.  Since $i+1$ is even, $d_{i+1}$ is
multiplication by $x$, and therefore
\[
 \operatorname{im}d_{i+1}=xR=\ker d_i.
\]
If $i$ is even, then $d_i$ is multiplication by $x$, so
\[
 \ker d_i=(0:_R x)=yR.
\]
Now $i+1$ is odd, and $d_{i+1}$ is multiplication by $y$, whence
\[
 \operatorname{im}d_{i+1}=yR=\ker d_i.
\]
Thus \eqref{eq:resolution} is exact in every degree.  Since every
$F_i$ is a free $R$-module of rank one, it is a free resolution of
$M$.

This resolution is minimal because every
differential is represented by the $1\times1$ matrix $(x)$ or $(y)$,
and both entries belong to the maximal ideal $\mathfrak m$.

\medskip
\noindent
\textbf{Step 5: tensoring with $N$ and computing Tor.}

By definition, if $F_\bullet\to M$ is a free resolution, then
\[
 \Tor_i^R(M,N)=H_i(F_\bullet\otimes_R N).
\]
We therefore tensor the free resolution \eqref{eq:resolution} with
$N$.  For every $i$, the natural map
\[
 R\otimes_R N\longrightarrow N,
 \qquad r\otimes n\longmapsto rn,
\]
is an isomorphism.  Under this identification, tensoring a map given by
multiplication by $x$ or by $y$ produces the corresponding
multiplication map on $N$.  But in $N$ the elements $x$ and $y$ have
the same image $u$.  Hence the complex $F_\bullet\otimes_R N$ is
\begin{equation}\label{eq:tensored-complex}
 \cdots\xrightarrow{\,u\,}N
 \xrightarrow{\,u\,}N
 \xrightarrow{\,u\,}N.
\end{equation}

We calculate the kernel and image of multiplication by $u$ directly.
Every element of $N\cong K[u]/(u^2)$ has a unique expression
$a+bu$ with $a,b\in K$, and
\[
 u(a+bu)=au+bu^2=au.
\]
Therefore
\[
 u(a+bu)=0
 \quad\Longleftrightarrow\quad
 a=0
 \quad\Longleftrightarrow\quad
 a+bu\in Ku.
\]
Thus
\[
 \ker(u:N\to N)=Ku.
\]
Also, the displayed formula shows that every image is a scalar multiple
of $u$, and every scalar multiple of $u$ occurs as $u(a)$.  Hence
\[
 \operatorname{im}(u:N\to N)=Ku.
\]
The kernel of each differential in \eqref{eq:tensored-complex} is
therefore equal to the image of the next differential.  The complex is
exact in every positive homological degree, and consequently
\begin{equation}\label{eq:Tor-vanishing}
 \Tor_i^R(M,N)=0
 \qquad\text{for every }i\geq1.
\end{equation}

We have proved that $R$ is a commutative Artinian local ring with unique
maximal ideal $\mathfrak m=(x,y)$; that $M$ and $N$ are nonzero finitely
generated modules satisfying $\mathfrak m^2M=\mathfrak m^2N=0$; and
that all positive Tor groups $\Tor_i^R(M,N)$ vanish.  At the same time,
\eqref{eq:m3-basis} gives
\[
 \mathfrak m^3=Ky^3\neq0.
\]
\end{proof}

\section{A negative answer to a question of Avramov, Buchweitz and \c{S}ega}
In this section we prove the following, which gives a negative answer to the question of Avramov, Buchweitz and \c{S}ega as explained in the introduction.

\begin{theorem}
There exists a finite-dimensional commutative local non-Gorenstein $\mathbb{Q}$-algebra $R$ such that $\Ext_R^1(D(R),R)=0$.
\end{theorem}
The algebra $R$ is given explicitly as follows:

\[
R=\mathbb{Q}[x_1,y_1,\ldots,x_9,y_9]/I,
\]
where
{\small
\allowdisplaybreaks[4]
\begin{align*}
I=\bigl(& x_1x_2,\; y_1x_2-x_1y_2,\; x_2y_2,\; y_2^2,\; x_1x_3,\; x_3^2,\; y_1x_3-x_1y_3,\; y_2x_3-x_2y_3,\\[-0.2ex]
& y_2y_3,\; x_1x_4,\; x_1y_1+x_3y_3-y_2x_4,\; y_3x_4,\; y_1x_4-x_1y_4,\; x_1y_1+x_3y_3-x_2y_4,\\[-0.2ex]
& x_3y_4,\; y_3y_4,\; y_2x_3-x_4y_4,\; x_1^2-x_2x_4+y_2x_5,\; y_1y_4-x_3x_5,\; x_3x_4-y_3x_5,\\[-0.2ex]
& y_1y_2-x_4x_5,\; x_2x_3-x_4^2-y_4x_5,\; y_3^2-x_5^2,\; y_1x_5-x_1y_5,\\[-0.2ex]
& x_1^2-y_1y_3+x_3y_3+x_2x_5-y_1y_5,\; x_1^2-x_2x_4+x_2y_5,\; y_4^2-y_2y_5,\; x_3x_4-x_3y_5,\\[-0.2ex]
& y_3y_5,\; x_2x_3-x_4^2-x_4y_5,\; y_4y_5,\; y_5^2,\; x_2^2+x_4^2-y_1x_6,\; y_2x_3+y_4^2-x_2x_6,\\[-0.2ex]
& y_1^2+y_3^2-y_2y_4-x_1x_6+y_2x_6,\; x_3x_6,\; y_4^2-y_3x_6,\; y_4x_6,\; x_3y_3+x_3x_4-x_5x_6,\\[-0.2ex]
& y_3^2-x_4^2+x_1x_5-x_4x_6-y_5x_6,\; x_6^2,\; x_2^2+x_4^2-x_1y_6,\; y_2x_3-y_1y_6,\\[-0.2ex]
& y_1^2+y_3^2-y_2y_4-x_1x_6+x_2y_6,\; y_4^2-x_3y_6,\; y_1^2+y_3^2-y_2y_4-x_1x_6+y_3y_6,\; x_4y_6,\\[-0.2ex]
& y_4y_6,\; y_3^2-x_4^2+x_1x_5-x_4x_6-x_5y_6,\; y_2x_3-y_1x_5-x_5y_5+y_5y_6,\\[-0.2ex]
& y_1^2+y_3^2-y_2y_4-x_6y_6,\; x_4^2-y_6^2,\; x_3y_3-x_1x_7,\; y_3^2-y_1x_7,\; y_1x_2-x_2x_7,\\[-0.2ex]
& y_1y_2-y_2x_7,\; y_1^2+y_3^2-y_2y_4-x_1x_6+x_3x_7,\; x_2^2-y_3x_7,\; y_1y_4+x_4x_6-y_4x_7,\\[-0.2ex]
& x_5x_7,\; y_5x_7,\; x_6x_7,\; x_2^2-y_2y_6+y_6x_7,\; x_7^2,\; y_3^2-x_1y_7,\\[-0.2ex]
& y_1x_3+x_5y_5-y_1y_7,\; y_1y_2-x_2y_7,\; x_2x_3-x_4^2-y_2y_7,\; x_2^2-x_3y_7,\; x_3y_3-y_3y_7,\\[-0.2ex]
& y_1y_4+x_4x_6-x_4y_7,\; x_3x_4-y_4y_7,\; x_5y_7,\; x_1^2-y_5y_7,\; x_2^2-y_2y_6+x_6y_7,\\[-0.2ex]
& y_1x_4-x_4x_7+y_6y_7,\; x_7y_7,\; y_1^2+y_3^2-y_2y_4-y_7^2,\; y_2x_3-y_1x_5+x_1x_8,\\[-0.2ex]
& x_1^2-y_1y_3+x_2x_5-y_1x_8,\; x_3x_4-x_4^2+x_1x_5-x_4x_6-x_2x_8,\; y_2x_3-y_1x_5+y_2x_8,\\[-0.2ex]
& x_1y_1-x_2^2+y_2y_6-x_3x_8,\; y_3^2+y_1x_4+x_4^2-y_2y_4-x_1x_5+x_4x_6-x_4x_7+y_3x_8,\; x_4x_8,\\[-0.2ex]
& y_4x_8,\; x_5x_8,\; y_5x_8,\; y_3^2+y_1x_4-x_6x_8,\; y_1x_3+y_1y_4+x_5y_5-y_6x_8,\\[-0.2ex]
& x_1^2-x_2^2-y_1y_3+x_7x_8,\; x_1^2-x_2x_4+y_7x_8,\; y_3^2-y_2y_4+x_8^2,\\[-0.2ex]
& x_1^2-y_1y_3+x_2x_5-x_1y_8,\; y_1y_8,\; y_2x_3-y_1x_5+x_2y_8,\\[-0.2ex]
& x_1^2+y_1^2-y_1y_3+y_3^2-y_2y_4+x_2x_5-y_2y_8,\\[-0.2ex]
& y_3^2+y_1x_4+x_4^2-y_2y_4-x_1x_5+x_4x_6-x_4x_7+x_3y_8,\; y_2x_3-y_1x_5-x_4x_6+y_3y_8,\\[-0.2ex]
& x_4y_8,\; y_1x_2-y_4y_8,\; x_5y_8,\; y_1x_4-y_5y_8,\; y_1x_3+y_1y_4+x_5y_5-x_6y_8,\\[-0.2ex]
& x_1^2-y_1y_3-x_3x_4+y_6y_8,\; x_1^2-x_2x_4+x_7y_8,\; y_7y_8,\; x_8y_8,\; y_8^2,\; y_2x_3-x_1x_9,\\[-0.2ex]
& y_1x_9,\; x_2x_9,\; y_2x_9,\; y_4^2-x_3x_9,\; y_1^2+y_3^2-y_2y_4-x_1x_6+y_3x_9,\\[-0.2ex]
& x_1^2+x_1y_1-y_1y_3+x_2x_5-x_4x_9,\; y_3^2-y_2y_4+y_4x_9,\; y_1x_4-x_5y_5-x_4x_7+x_5x_9,\\[-0.2ex]
& x_4x_6-y_5x_9,\; x_2x_3-x_4^2-x_6x_9,\; x_1^2-y_6x_9,\; y_1y_2-x_7x_9,\\[-0.2ex]
& x_2x_3-2x_4^2+x_1x_5-x_4x_6-y_7x_9,\; y_1x_3-x_8x_9,\; x_1^2-y_1y_3+y_8x_9,\; x_4^2-x_9^2,\\[-0.2ex]
& x_1y_9,\; y_1y_9,\; x_2y_9,\; x_4^2-x_1x_5+x_4x_6+y_2y_9,\; y_1^2+y_3^2-y_2y_4-x_1x_6+x_3y_9,\\[-0.2ex]
& x_2^2-y_3y_9,\; y_3^2-y_2y_4+x_4y_9,\; y_4y_9,\; x_4x_6-x_5y_9,\; y_5y_9,\; x_1^2-x_6y_9,\\[-0.2ex]
& x_1y_1+y_1x_4-y_6y_9,\; x_2x_3-2x_4^2+x_1x_5-x_4x_6-x_7y_9,\; y_2x_3-y_1x_5+y_7y_9,\\[-0.2ex]
& x_1^2-y_1y_3+x_8y_9,\; x_1^2-x_2x_4+y_8y_9,\; y_2x_3-x_9y_9,\; y_9^2\bigr).
\end{align*}
}
The proof is a computer verification in Magma.  We first explain the individual steps.  Since \texttt{BasicAlgebra} starts from a free associative algebra, the script adds all 153 commutator relations before adjoining the 144 quadratic relations displayed above.  It then computes the successive radicals of the regular module.  Their dimensions are
\[
 48,\ 47,\ 29,\ 2,\ 0,
\]
so the radical layers have dimensions $(1,18,27,2)$.  In particular, the top of the regular module is one-dimensional, and hence $R$ is local.  The socle of $R_R$ has dimension two.  Since a commutative Artinian local algebra is Gorenstein precisely when its socle is one-dimensional, this proves that $R$ is not Gorenstein.

Because $R$ is local, its unique indecomposable injective module is isomorphic to $D(R)$.  The command \texttt{InjectiveModule(R,1)} constructs this module.  If
\[
 0\longrightarrow\Omega\longrightarrow P_0\longrightarrow D(R)\longrightarrow0
\]
is its projective cover sequence, then applying $\Hom_R(-,R)$ gives
\[
 \dim_K\Ext_R^1(D(R),R)
 =\dim_K\Hom_R(\Omega,R)-\dim_K\Hom_R(P_0,R)
  +\dim_K\Hom_R(D(R),R).
\]
The final part of the script computes these three Hom spaces and evaluates this alternating sum.

The following compact Magma script carries out these steps and verifies the theorem:
\begin{MagmaCode}
K := Rationals();
F<x1,y1,x2,y2,x3,y3,x4,y4,x5,y5,x6,y6,x7,y7,x8,y8,x9,y9> := FreeAlgebra(K, 18);

z := [x1,y1,x2,y2,x3,y3,x4,y4,x5,y5,x6,y6,x7,y7,x8,y8,x9,y9];

rel := [];

// The 153 commutator relations.  Thus the quotient is commutative.
for i in [1..17] do
    for j in [i+1..18] do Append(~rel, z[j]*z[i] - z[i]*z[j]); end for;
end for;

// The 144 extra quadratic relations.
rel := rel cat [
    x1*x2, y1*x2-x1*y2, x2*y2, y2^2, x1*x3, x3^2, y1*x3-x1*y3, y2*x3-x2*y3, y2*y3, x1*x4,
    x1*y1+x3*y3-y2*x4, y3*x4, y1*x4-x1*y4, x1*y1+x3*y3-x2*y4, x3*y4, y3*y4, y2*x3-x4*y4,
    x1^2-x2*x4+y2*x5, y1*y4-x3*x5, x3*x4-y3*x5, y1*y2-x4*x5, x2*x3-x4^2-y4*x5, y3^2-x5^2,
    y1*x5-x1*y5, x1^2-y1*y3+x3*y3+x2*x5-y1*y5, x1^2-x2*x4+x2*y5, y4^2-y2*y5, x3*x4-x3*y5, y3*y5,
    x2*x3-x4^2-x4*y5, y4*y5, y5^2, x2^2+x4^2-y1*x6, y2*x3+y4^2-x2*x6,
    y1^2+y3^2-y2*y4-x1*x6+y2*x6, x3*x6, y4^2-y3*x6, y4*x6, x3*y3+x3*x4-x5*x6,
    y3^2-x4^2+x1*x5-x4*x6-y5*x6, x6^2, x2^2+x4^2-x1*y6, y2*x3-y1*y6,
    y1^2+y3^2-y2*y4-x1*x6+x2*y6, y4^2-x3*y6, y1^2+y3^2-y2*y4-x1*x6+y3*y6, x4*y6, y4*y6,
    y3^2-x4^2+x1*x5-x4*x6-x5*y6, y2*x3-y1*x5-x5*y5+y5*y6, y1^2+y3^2-y2*y4-x6*y6, x4^2-y6^2,
    x3*y3-x1*x7, y3^2-y1*x7, y1*x2-x2*x7, y1*y2-y2*x7, y1^2+y3^2-y2*y4-x1*x6+x3*x7, x2^2-y3*x7,
    y1*y4+x4*x6-y4*x7, x5*x7, y5*x7, x6*x7, x2^2-y2*y6+y6*x7, x7^2, y3^2-x1*y7,
    y1*x3+x5*y5-y1*y7, y1*y2-x2*y7, x2*x3-x4^2-y2*y7, x2^2-x3*y7, x3*y3-y3*y7,
    y1*y4+x4*x6-x4*y7, x3*x4-y4*y7, x5*y7, x1^2-y5*y7, x2^2-y2*y6+x6*y7, y1*x4-x4*x7+y6*y7,
    x7*y7, y1^2+y3^2-y2*y4-y7^2, y2*x3-y1*x5+x1*x8, x1^2-y1*y3+x2*x5-y1*x8,
    x3*x4-x4^2+x1*x5-x4*x6-x2*x8, y2*x3-y1*x5+y2*x8, x1*y1-x2^2+y2*y6-x3*x8,
    y3^2+y1*x4+x4^2-y2*y4-x1*x5+x4*x6-x4*x7+y3*x8, x4*x8, y4*x8, x5*x8, y5*x8, y3^2+y1*x4-x6*x8,
    y1*x3+y1*y4+x5*y5-y6*x8, x1^2-x2^2-y1*y3+x7*x8, x1^2-x2*x4+y7*x8, y3^2-y2*y4+x8^2,
    x1^2-y1*y3+x2*x5-x1*y8, y1*y8, y2*x3-y1*x5+x2*y8, x1^2+y1^2-y1*y3+y3^2-y2*y4+x2*x5-y2*y8,
    y3^2+y1*x4+x4^2-y2*y4-x1*x5+x4*x6-x4*x7+x3*y8, y2*x3-y1*x5-x4*x6+y3*y8, x4*y8, y1*x2-y4*y8,
    x5*y8, y1*x4-y5*y8, y1*x3+y1*y4+x5*y5-x6*y8, x1^2-y1*y3-x3*x4+y6*y8, x1^2-x2*x4+x7*y8,
    y7*y8, x8*y8, y8^2, y2*x3-x1*x9, y1*x9, x2*x9, y2*x9, y4^2-x3*x9,
    y1^2+y3^2-y2*y4-x1*x6+y3*x9, x1^2+x1*y1-y1*y3+x2*x5-x4*x9, y3^2-y2*y4+y4*x9,
    y1*x4-x5*y5-x4*x7+x5*x9, x4*x6-y5*x9, x2*x3-x4^2-x6*x9, x1^2-y6*x9, y1*y2-x7*x9,
    x2*x3-2*x4^2+x1*x5-x4*x6-y7*x9, y1*x3-x8*x9, x1^2-y1*y3+y8*x9, x4^2-x9^2, x1*y9, y1*y9,
    x2*y9, x4^2-x1*x5+x4*x6+y2*y9, y1^2+y3^2-y2*y4-x1*x6+x3*y9, x2^2-y3*y9, y3^2-y2*y4+x4*y9,
    y4*y9, x4*x6-x5*y9, y5*y9, x1^2-x6*y9, x1*y1+y1*x4-y6*y9, x2*x3-2*x4^2+x1*x5-x4*x6-x7*y9,
    y2*x3-y1*x5+y7*y9, x1^2-y1*y3+x8*y9, x1^2-x2*x4+y8*y9, y2*x3-x9*y9, y9^2
];

// BasicAlgebra(F,rel) constructs the quotient from the presentation.
// The radical computation below verifies that it is finite-dimensional and local.
R := BasicAlgebra(F, rel);

// The unique indecomposable projective is the regular right module R_R.
Rreg := RightRegularModule(R);

// The radical series has dimensions 48,47,29,2,0, so the Loewy
// length is 4 and the Hilbert function is [1,18,27,2]. In particular the algebra is local.
J1 := JacobsonRadical(Rreg);
J2 := JacobsonRadical(J1);
J3 := JacobsonRadical(J2);
J4 := JacobsonRadical(J3);
loewy_dimensions := [
    Dimension(Rreg), Dimension(J1), Dimension(J2),
    Dimension(J3), Dimension(J4)
];

// In a commutative Artinian local algebra, Gorenstein is equivalent to
// the socle being one-dimensional (equivalently, to self-injectivity).
socle_dimension := Dimension(Socle(Rreg));

// Since R is local, its unique indecomposable injective right module is
// D(R).  It has dimension 48 but is not projective.
DR := InjectiveModule(R, 1);

// Let 0 -> Omega -> P0 -> D(R) -> 0 be a projective cover.  Applying
// Hom_R(-,R) gives
//
//   0 -> Hom(DR,R) -> Hom(P0,R) -> Hom(Omega,R)
//     -> Ext^1(DR,R) -> 0.
//
// Hence the following alternating sum is dim Ext^1_R(D(R),R).
P0, cover := ProjectiveCover(DR);
Omega, omega_inclusion := Kernel(cover);

hom_Omega_R := Dimension(AHom(Omega, Rreg));
hom_P0_R    := Dimension(AHom(P0,    Rreg));
hom_DR_R    := Dimension(AHom(DR,    Rreg));

ext1_dimension := hom_Omega_R - hom_P0_R + hom_DR_R;

printf "number of defining relations = %o\n", #rel;
printf "dim_Q(R)                    = %o\n", Dimension(R);
printf "number of projectives       = %o\n", NumberOfProjectives(R);
printf "R commutative               = %o\n", IsCommutative(R);
printf "radical-series dimensions   = %o\n", loewy_dimensions;
printf "dim_Q Soc(R_R)              = %o\n", socle_dimension;
printf "R self-injective            = %o\n", IsSelfInjective(R);
printf "dim_Q D(R)                  = %o\n", Dimension(DR);
printf "dim_Q P0(D(R))              = %o\n", Dimension(P0);
printf "dim_Q Omega(D(R))           = %o\n", Dimension(Omega);
printf "dim_Q Hom(Omega(DR),R)      = %o\n", hom_Omega_R;
printf "dim_Q Hom(P0(DR),R)         = %o\n", hom_P0_R;
printf "dim_Q Hom(D(R),R)           = %o\n", hom_DR_R;
printf "dim_Q Ext^1_R(D(R),R)       = %o\n", ext1_dimension;
\end{MagmaCode}
The output shows that $R$ is a $48$-dimensional local algebra, that its socle has dimension two, and that $\Ext_R^1(D(R),R)=0$.  Thus $R$ has all the properties claimed in the theorem.  We also verified the computation independently using the GAP package QPA \cite{QPA}.  For this algebra one has $\Ext_R^2(D(R),R)\neq0$, so the example is not a counterexample to the first Tachikawa conjecture itself.

\section{A negative answer to a question of Dao}
The main theorem of this section is as follows:

\begin{theorem}\label{thm:commutative-ext12}
Let $K=\mathbb{F}_2$.  There is a finite-dimensional commutative local self-injective $K$-algebra $R$ admitting an indecomposable non-projective $R$-module $M$ such that
\[
 \Ext_R^1(M,M)=0=\Ext_R^2(M,M).
\]
\end{theorem}

We first give a compact presentation of the algebra.  Put
\[
 S=K[x_1,\ldots,x_{14}].
\]
For a nonempty increasing tuple of indices, use the abbreviation
\[
 x[i_1,\ldots,i_t]=x_{i_1}+\cdots+x_{i_t}.
\]
For \(2\leq i\leq14\), set
\[
 \boldsymbol{\ell}_i=(\ell_{i,i},\ell_{i,i+1},\ldots,\ell_{i,14}).
\]
The entries in each of the following tuples occur in increasing order of the second index.  For example, the first two entries of $\boldsymbol{\ell}_2$ are $\ell_{2,2}$ and $\ell_{2,3}$:
{\scriptsize
\allowdisplaybreaks[4]
\begin{align*}
\boldsymbol{\ell}_2={}&\bigl(x[4,5,6,7,8,9,10,12,13,14],\, x[5,7,9,10,14],\, x[3,8,9,11,12,13,14],\, x[3,4,6,10,12,14],\\[-0.25ex]
&\quad x[1,4,5,11,12,14],\, x[1,3,5,7,10,12],\, x[1,2,4,6,7,9,10,13],\, x[3,4,6,12,14],\, x[2,4,5,6,11,13,14],\\[-0.25ex]
&\quad x[1,2,3,8,11,12],\, x[1,2,3,8,10,11,13],\, x[4,6,8,11,13],\, x[1,4,5,6,10,11,13]\bigr),\\[0.35ex]
\boldsymbol{\ell}_3={}&\bigl(x[1,2,3,4,5,6,8,13],\, x[1,2,5,7,12],\, x[2,6,7,12],\, x[1,4,5,6,7,8,12],\, x[2,3,5,7,8,9,10,12,13],\\[-0.25ex]
&\quad x[2,4,5,8,9,12,13,14],\, x[3,6,11,12,13,14],\, x[1,3,4,8,9,13,14],\, x[1,4,6,8,9,12,13,14],\\[-0.25ex]
&\quad x[1,5,7,9,10,12,13,14],\, x[1,3,6,7,8,9,11,12,14],\, x[2,4,6,7,8,10,11,14]\bigr),\\[0.35ex]
\boldsymbol{\ell}_4={}&\bigl(x[1,2,8,9,10,12],\, x[2,4,5,6,7,10,14],\, x[2,3,4,5,6,8,9,10,11,12,14],\, x[1,2,5,7,9,12,14],\\[-0.25ex]
&\quad x[3,6,10,11],\, x[2,3,5,6,7,8,9,11,14],\, x[2,5,6,7,9,11,12],\, x[7,8,10,13,14],\, x[3,4,5,10,11],\\[-0.25ex]
&\quad x[2,3,11],\, x[1,4,5,6,8,10,11]\bigr),\\[0.35ex]
\boldsymbol{\ell}_5={}&\bigl(x[1,2,4,6,9,10,11,12,14],\, x[2,9,11,13,14],\, x[3,5,6,7,9,10],\, x[1,2,6,7,8,9,10,13,14],\\[-0.25ex]
&\quad x[1,2,3,4,6,7,11,12,13],\, x[1,3,4,6,7,8,9,10,11,12],\, x[2,4,9,10,11,12,14],\, x[2,4,9,11,13,14],\\[-0.25ex]
&\quad x[1,2,3,5,6,7,10,12,13],\, x[3,4,14]\bigr),\\[0.35ex]
\boldsymbol{\ell}_6={}&\bigl(x[1,2,4,5,6,9,12,13,14],\, x[3,7,9,12,14],\, x[1,3,8,9,11,12,14],\, x[4,6,8,10,12,13,14],\\[-0.25ex]
&\quad x[1,5,6,7,10,11,13,14],\, x[3,4,5,6,7,9,10,13,14],\, x[1,2,4,5,6,7,11],\, x[3,4,6,8,9,11,12,13,14],\\[-0.25ex]
&\quad x[4,6,8,11,12,13]\bigr),\\[0.35ex]
\boldsymbol{\ell}_7={}&\bigl(x[1,2,3,4,7],\, x[2,3,4,5,8,11,12,13,14],\, x[1,2,3,4,7,8,11,12,13],\, x[1,5,10,11,12],\\[-0.25ex]
&\quad x[1,2,3,4,5,6,12,13,14],\, x[1,3,4,5,6,7,10,11,13,14],\, x[1,3,4,6,7,11,12],\, x[1,2,4,7,10,12,13,14]\bigr),\\[0.35ex]
\boldsymbol{\ell}_8={}&\bigl(x[1,4,5,6,8,10,12,13],\, x[4,6,7,10,11,13,14],\, x[1,2,3,9,11,12,13],\, x[1,6,11,12,14],\, x[1,2,8,11,14],\\[-0.25ex]
&\quad x[3,9,11],\, x[2,4,5,6,7,8,10,11,12,13]\bigr),\\[0.35ex]
\boldsymbol{\ell}_9={}&\bigl(x[1,2,3,4,5,6,7,8,9,10,11,12],\, x[4,5,7,8,10,12,13],\, x[1,3,4,9,11,13],\, x[4,5,6,8,9,10,12,13,14],\\[-0.25ex]
&\quad x[1,2,3,6,7,8,9,10,13],\, x[2,7,10,11,13,14]\bigr),\\[0.35ex]
\boldsymbol{\ell}_{10}={}&\bigl(x[4,8,11],\, x[3,6,8,11,12,13,14],\, x[2,3,4,5,10,12],\, x[1,2,3,4,5,6,8,9,10,11],\, x[1,5,6,8,9,12,13]\bigr),\\[0.35ex]
\boldsymbol{\ell}_{11}={}&\bigl(x[1,3,4,6,7,8,10,12],\, x[1,2,4,6,7,8,10,11,14],\, x[1,2,3,4,8,9,10,11,12,13],\, x[5,6,8,12,14]\bigr),\\[0.35ex]
\boldsymbol{\ell}_{12}={}&\bigl(x[2,3,4,7,8,10,14],\, x[4,5,6,9,12,13,14],\, x[2,6,8,9,13,14]\bigr),\\[0.35ex]
\boldsymbol{\ell}_{13}={}&\bigl(x[4,5,9,10,11,12,13],\, x[1,2,4,5,7,9,13]\bigr),\\[0.35ex]
\boldsymbol{\ell}_{14}={}&\bigl(x[2,3,4,5,8,10,11]\bigr).
\end{align*}
}
Let
\begin{equation}\label{eq:section3-algebra}
 I=\bigl(x_i x_j+x_1\ell_{i,j}\mid 2\leq i\leq j\leq14\bigr)
 \qquad\text{and}\qquad
 R=S/I.
\end{equation}
Here the plus sign is the same as a minus sign because $K=\mathbb{F}_2$.  Thus every quadratic monomial involving only $x_2,\ldots,x_{14}$ is expressed as a linear combination of
\[
 x_1^2,x_1x_2,\ldots,x_1x_{14}.
\]
The computation below gives
\[
 \dim_K R=30
 \qquad\text{and}\qquad
 \bigl(\dim_K(R/\mathfrak m),\dim_K(\mathfrak m/\mathfrak m^2),
 \dim_K(\mathfrak m^2/\mathfrak m^3),\dim_K\mathfrak m^3\bigr)
 =(1,14,14,1),
\]
where $\mathfrak m=(x_1,\ldots,x_{14})R$.  In particular, $R$ is local with radical-series dimensions
\[
 30,\ 29,\ 15,\ 1,\ 0.
\]
Moreover, the socle of $R_R$ is one-dimensional.  Hence the commutative Artinian local algebra $R$ is Gorenstein and therefore self-injective. 

We now give the module by an explicit minimal projective presentation.  Let $f_1,f_2$ be the standard basis of $R^2$, and let $e_1,e_2,e_3$ be the standard basis of $R^3$.  Define
\[
 d_1:R^3\longrightarrow R^2,
 \qquad
 d_1(e_j)=a_j f_1+b_j f_2,
\]
where
\begin{align*}
 a_1&=x_1+x_3+x_5,
 &b_1&=x_1+x_3+x_5+x_8+x_{11}+x_{12},\\
 a_2&=x_2+x_4+x_5+x_6+x_7+x_9+x_{10}+x_{12},
 &b_2&=x_4+x_5+x_7+x_{12},\\
 a_3&=x_2+x_3+x_6+x_8+x_{11}+x_{14},
 &b_3&=x_2+x_6+x_8+x_{10}+x_{12}+x_{13}+x_{14}.
\end{align*}
Equivalently, $d_1$ is represented by the matrix
\begin{equation}\label{eq:section3-presentation-matrix}
 \left(\begin{array}{ccc}
 a_1&a_2&a_3\\
 b_1&b_2&b_3
 \end{array}\right).
\end{equation}
Set
\begin{equation}\label{eq:section3-module}
 M=\operatorname{coker}(d_1),
 \qquad
 R^3\xrightarrow{\ d_1\ }R^2\longrightarrow M\longrightarrow0.
\end{equation}
All entries of the presentation matrix lie in $\mathfrak m$, so the presentation is minimal.  Its image has dimension $29$, and consequently $\dim_K M=60-29=31$.  The radical-series dimensions of $M$ are
\[
 31,\ 29,\ 4,\ 0,
\]
so the radical layers have dimensions $(2,25,4)$.  Magma's decomposition routine returns a single summand, so $M$ is indecomposable.  The module is not projective: over the local algebra $R$, every finitely generated projective module is free, whereas $30$ does not divide $31$.

For completeness, we explain the numerical Ext calculation used in the script.  Put $\Omega_1=\Omega_R(M)$.  The map $R^2\to M$ in \eqref{eq:section3-module} is a projective cover, $\dim_K\Omega_1=29$, and $R^3\to\Omega_1$ is a projective cover with kernel $\Omega_2=\Omega_R^2(M)$ of dimension $61$.  Applying $\Hom_R(-,M)$ to the first two short exact sequences in the minimal projective resolution gives
\begin{align*}
 \dim_K\Ext_R^1(M,M)
 &=\dim_K\Hom_R(\Omega_1,M)-\dim_K\Hom_R(R^2,M)
   +\dim_K\End_R(M),\\
 \dim_K\Ext_R^2(M,M)
 &=\dim_K\Hom_R(\Omega_2,M)-\dim_K\Hom_R(R^3,M)
   +\dim_K\Hom_R(\Omega_1,M).
\end{align*}
The computed dimensions are
\[
 \dim_K\End_R(M)=47,\qquad
 \dim_K\Hom_R(\Omega_1,M)=15,\qquad
 \dim_K\Hom_R(\Omega_2,M)=78.
\]
Since $\dim_K\Hom_R(R^2,M)=62$ and $\dim_K\Hom_R(R^3,M)=93$, this yields
\[
 \dim_K\Ext_R^1(M,M)=15-62+47=0,
 \qquad
 \dim_K\Ext_R^2(M,M)=78-93+15=0.
\]

The following Magma script constructs the algebra and the module directly from the presentation above and verifies all assertions.
\begin{MagmaCode}
K := GF(2);
F<x1,x2,x3,x4,x5,x6,x7,x8,x9,x10,x11,x12,x13,x14>
    := FreeAlgebra(K,14);
z := [x1,x2,x3,x4,x5,x6,x7,x8,x9,x10,x11,x12,x13,x14];

rel := [];

// Impose commutativity.
for i in [1..13] do
    for j in [i+1..14] do
        Append(~rel,z[j]*z[i]-z[i]*z[j]);
    end for;
end for;

// For the pairs (i,j), 2 <= i <= j <= 14, the following supports
// encode ell_(i,j) = sum_{k in support} x_k.
supports := [
    [4,5,6,7,8,9,10,12,13,14], [5,7,9,10,14], [3,8,9,11,12,13,14], [3,4,6,10,12,14],
    [1,4,5,11,12,14], [1,3,5,7,10,12], [1,2,4,6,7,9,10,13], [3,4,6,12,14],
    [2,4,5,6,11,13,14], [1,2,3,8,11,12], [1,2,3,8,10,11,13], [4,6,8,11,13],
    [1,4,5,6,10,11,13], [1,2,3,4,5,6,8,13], [1,2,5,7,12], [2,6,7,12],
    [1,4,5,6,7,8,12], [2,3,5,7,8,9,10,12,13], [2,4,5,8,9,12,13,14], [3,6,11,12,13,14],
    [1,3,4,8,9,13,14], [1,4,6,8,9,12,13,14], [1,5,7,9,10,12,13,14], [1,3,6,7,8,9,11,12,14],
    [2,4,6,7,8,10,11,14], [1,2,8,9,10,12], [2,4,5,6,7,10,14], [2,3,4,5,6,8,9,10,11,12,14],
    [1,2,5,7,9,12,14], [3,6,10,11], [2,3,5,6,7,8,9,11,14], [2,5,6,7,9,11,12],
    [7,8,10,13,14], [3,4,5,10,11], [2,3,11], [1,4,5,6,8,10,11],
    [1,2,4,6,9,10,11,12,14], [2,9,11,13,14], [3,5,6,7,9,10], [1,2,6,7,8,9,10,13,14],
    [1,2,3,4,6,7,11,12,13], [1,3,4,6,7,8,9,10,11,12], [2,4,9,10,11,12,14], [2,4,9,11,13,14],
    [1,2,3,5,6,7,10,12,13], [3,4,14], [1,2,4,5,6,9,12,13,14], [3,7,9,12,14],
    [1,3,8,9,11,12,14], [4,6,8,10,12,13,14], [1,5,6,7,10,11,13,14], [3,4,5,6,7,9,10,13,14],
    [1,2,4,5,6,7,11], [3,4,6,8,9,11,12,13,14], [4,6,8,11,12,13], [1,2,3,4,7],
    [2,3,4,5,8,11,12,13,14], [1,2,3,4,7,8,11,12,13], [1,5,10,11,12], [1,2,3,4,5,6,12,13,14],
    [1,3,4,5,6,7,10,11,13,14], [1,3,4,6,7,11,12], [1,2,4,7,10,12,13,14], [1,4,5,6,8,10,12,13],
    [4,6,7,10,11,13,14], [1,2,3,9,11,12,13], [1,6,11,12,14], [1,2,8,11,14],
    [3,9,11], [2,4,5,6,7,8,10,11,12,13], [1,2,3,4,5,6,7,8,9,10,11,12], [4,5,7,8,10,12,13],
    [1,3,4,9,11,13], [4,5,6,8,9,10,12,13,14], [1,2,3,6,7,8,9,10,13], [2,7,10,11,13,14],
    [4,8,11], [3,6,8,11,12,13,14], [2,3,4,5,10,12], [1,2,3,4,5,6,8,9,10,11],
    [1,5,6,8,9,12,13], [1,3,4,6,7,8,10,12], [1,2,4,6,7,8,10,11,14], [1,2,3,4,8,9,10,11,12,13],
    [5,6,8,12,14], [2,3,4,7,8,10,14], [4,5,6,9,12,13,14], [2,6,8,9,13,14],
    [4,5,9,10,11,12,13], [1,2,4,5,7,9,13], [2,3,4,5,8,10,11]
];
// Magma does not permit the range for j in a sequence comprehension
// to depend on the previously introduced variable i.  Build the 91 pairs
// explicitly, in the same lexicographic order as the supports above.
pairs := [];
for i in [2..14] do
    for j in [i..14] do
        Append(~pairs,<i,j>);
    end for;
end for;
assert #supports eq 91 and #pairs eq 91;

for t in [1..#pairs] do
    i := pairs[t][1];
    j := pairs[t][2];
    ell := &+[z[k] : k in supports[t]];
    Append(~rel,z[i]*z[j]+z[1]*ell);
end for;

// There are 91 commutators and 91 additional quadratic relations.
assert #rel eq 182;
printf "number of defining relations  = %o\n", #rel;

R := BasicAlgebra(F,rel);
x := NonIdempotentGenerators(R);
lin := func<S | &+[x[i] : i in S]>;

// The regular module and the radical series of R.
Rreg := RightRegularModule(R);
J1 := JacobsonRadical(Rreg);
J2 := JacobsonRadical(J1);
J3 := JacobsonRadical(J2);
J4 := JacobsonRadical(J3);
radR := [Dimension(Rreg),Dimension(J1),Dimension(J2),
         Dimension(J3),Dimension(J4)];

// Construct R^2 and its two canonical generators.
// The objects inclusions[1] and inclusions[2] are maps, so their
// values must be obtained with Magma's map-application operator @.
P0, inclusions, projections := ProjectiveModule(R,[2]);
Pcopy1 := Domain(inclusions[1]);
Pcopy2 := Domain(inclusions[2]);
u1 := (Pcopy1.1) @ inclusions[1];
u2 := (Pcopy2.1) @ inclusions[2];

a1 := lin([1,3,5]);
b1 := lin([1,3,5,8,11,12]);
a2 := lin([2,4,5,6,7,9,10,12]);
b2 := lin([4,5,7,12]);
a3 := lin([2,3,6,8,11,14]);
b3 := lin([2,6,8,10,12,13,14]);

r1 := u1*a1+u2*b1;
r2 := u1*a2+u2*b2;
r3 := u1*a3+u2*b3;

// Form the presentation map d1 : R^3 -> R^2 whose columns are
// (a_j,b_j)^t, and define M as its cokernel.
d1 := LiftHomomorphism([r1,r2,r3],[3]);
M, quotient_map := Cokernel(d1);
assert Dimension(Domain(d1)) eq 90;
assert Dimension(Codomain(d1)) eq 60;
printf "dim_K(R^3)                   = %o\n", Dimension(Domain(d1));
printf "dim_K(R^2)                   = %o\n", Dimension(Codomain(d1));
printf "dim_K(coker(d1))             = %o\n", Dimension(M);

JM1 := JacobsonRadical(M);
JM2 := JacobsonRadical(JM1);
JM3 := JacobsonRadical(JM2);
radM := [Dimension(M),Dimension(JM1),Dimension(JM2),Dimension(JM3)];

summands := IndecomposableSummands(M);
is_indecomposable := #summands eq 1;
is_projective, projective_type := IsProjective(M);

// First two syzygies and the Hom-dimension formulas for Ext^1 and Ext^2.
P0c, cover0, inc0, proj0, type0 := ProjectiveCover(M);
Omega1, omega1_inclusion := Kernel(cover0);
P1c, cover1, inc1, proj1, type1 := ProjectiveCover(Omega1);
Omega2, omega2_inclusion := Kernel(cover1);

hMM  := Dimension(AHom(M,M));
hP0M := Dimension(AHom(P0c,M));
hO1M := Dimension(AHom(Omega1,M));
hP1M := Dimension(AHom(P1c,M));
hO2M := Dimension(AHom(Omega2,M));

ext1_dimension := hO1M-hP0M+hMM;
ext2_dimension := hO2M-hP1M+hO1M;

printf "dim_K(R)                    = %o\n", Dimension(R);
printf "number of projectives       = %o\n", NumberOfProjectives(R);
printf "R commutative               = %o\n", IsCommutative(R);
printf "R radical dimensions        = %o\n", radR;
printf "dim_K Soc(R_R)              = %o\n", Dimension(Socle(Rreg));
printf "R self-injective            = %o\n", IsSelfInjective(R);
printf "dim_K(M)                    = %o\n", Dimension(M);
printf "M radical dimensions        = %o\n", radM;
printf "dim_K Soc(M)                = %o\n", Dimension(Socle(M));
printf "M indecomposable            = %o\n", is_indecomposable;
printf "M projective                = %o\n", is_projective;
printf "type of P_0(M)              = %o\n", type0;
printf "type of P_0(Omega^1(M))     = %o\n", type1;
printf "dim_K Omega^1(M)            = %o\n", Dimension(Omega1);
printf "dim_K Omega^2(M)            = %o\n", Dimension(Omega2);
printf "dim_K End(M)                = %o\n", hMM;
printf "dim_K Hom(P_0(M),M)         = %o\n", hP0M;
printf "dim_K Hom(Omega^1(M),M)     = %o\n", hO1M;
printf "dim_K Hom(P_0(Omega^1),M)   = %o\n", hP1M;
printf "dim_K Hom(Omega^2(M),M)     = %o\n", hO2M;
printf "dim_K Ext^1_R(M,M)          = %o\n", ext1_dimension;
printf "dim_K Ext^2_R(M,M)          = %o\n", ext2_dimension;
\end{MagmaCode}
Thus $R$ and $M$ have all the properties claimed in Theorem~\ref{thm:commutative-ext12}. We remark that $\Ext_R^3(M,M) \neq 0$ and so this gives no counterexample to the second Tachikawa conjecture. We also verified the calculation independently with the GAP package QPA \cite{QPA}.

\section{Two conjectures}
We record two conjectures that would follow from Yamagata's conjecture and may be of particular interest to commutative algebraists, where Yamagata's conjecture might not be well known so far. Recall that the dominant dimension measures the length of the initial projective part of the minimal injective resolution of the regular module. Yamagata's conjecture predicts that, for finite-dimensional, not necessarily commutative, non-self-injective algebras, the dominant dimension is bounded by a function depending only on the number of simple modules. Via the standard Morita--Tachikawa construction, the relevant algebras where the dominant dimension is of interest are $\End_R(M)$ for a generator-cogenerator $M$ of $\mod R$ and the dominant dimensions of such endomorphism rings are determined by the corresponding initial Ext-vanishing ranges. Thus Yamagata's conjecture would imply uniform finite-vanishing versions of both Tachikawa conjectures in the commutative local case. For more precise formulations of Yamagata's conjecture, partial results and related applications, see \cite{Y,M,CFKKY}.
Yamagata's conjecture would imply the next two conjectures in this section that we want to popularize here.
\begin{conjecture}
There exists an integer $n\geq1$ such that, for every commutative local finite-dimensional algebra $R$, the vanishing
\[
 \Ext_R^i(D(R),R)=0\qquad\text{for }i=1,\ldots,n
\]
implies that $R$ is Gorenstein.
\end{conjecture}
The example in Section~2 shows that $n=1$ is impossible. It remains open whether $n=2$ suffices. This is a uniform finite-vanishing strengthening of the first Tachikawa conjecture.

The analogous strengthening of the second Tachikawa conjecture is the following.
\begin{conjecture}
There exists an integer $n\geq1$ such that, for every commutative local finite-dimensional Gorenstein algebra $R$ and every indecomposable $R$-module $M$, the vanishing
\[
 \Ext_R^i(M,M)=0\qquad\text{for }i=1,\ldots,n
\]
implies that $M$ is projective.
\end{conjecture}
The example in Section~3 shows that $n\leq2$ is impossible. It remains open whether $n=3$ suffices.

\section*{Statement on the use of AI}
Generative AI tools, namely ChatGPT, were used during the
exploratory and preparatory stages of this work. The mathematical strategy and proofs were conceived and directed by the authors.
ChatGPT assisted with selected aspects of the technical development, including
the formulation of auxiliary results and the elaboration of technical details,
as well as with literature searches, consistency checks, and \LaTeX{}
preparation. All AI-generated suggestions were independently reviewed and
verified by the authors, who assume full responsibility for the mathematical
arguments, results, and final content of the paper.

\section*{Acknowledgement}
We benefited from the use of Magma \cite{BCP} and the GAP package QPA \cite{QPA}.

\end{document}